\documentclass{mscs}

\usepackage{amsfonts,amssymb,epsfig,times, amsmath}
\usepackage{pstricks}
\usepackage{stmaryrd}
\usepackage{graphicx}
\usepackage{epstopdf}

\title[Garside monoids {\it vs} divisibility monoids]
  {Garside monoids {\it vs} divisibility monoids}
\author[Matthieu Picantin]
  {M\ls A\ls T\ls T\ls H\ls I\ls E\ls U\ns
   P\ls I\ls C\ls A\ls N\ls T\ls I\ls N
\thanks{The author thanks Prof. M. Droste and D. Kuske from the Institut f\" ur Algebra, TU
Dresden, Germany, where he wrote this paper. He is grateful to the two anonymous referees for
comments on an earlier draft.}\\
Laboratoire Nicolas Oresme, UMR 6139 CNRS\\
Universit\'e de Caen, F-14000 Caen\\
Email: picantin@math.unicaen.fr}

\date{31 January 2003; revised 12 January 2004}
\pubyear{2004}
\pagerange{\pageref{firstpage}--\pageref{lastpage}}
\volume{00}

\newtheorem{prop}{Proposition}[section]
\newtheorem{lemma}[prop]{Lemma}
\newtheorem{cor}[prop]{Corollary}
\newtheorem{thm}[prop]{Theorem}
\newtheorem{exam}[prop]{Example}
\newtheorem{nota}[prop]{Notation}
\newtheorem{defn}[prop]{Definition}
\newtheorem{rem}[prop]{Remark}

\let\D=\Delta
\def\Dar{\hbox{$\downarrow$}}

\def\ie{{\it i.e.}}

\def\resp{resp\hbox{. }}
\def\dR{\mathord{\setminus}}

\def\gL{\mathbin{\scriptstyle\wedge}}

\def\jR{\mathbin{{\scriptstyle\vee}}}

\def\jjR{\mathbin{\bigvee}}
\def\zR#1{{\Delta_{#1}}}
\def\zL#1{\widetilde{\Delta}_{#1}}

\def\D{\Delta}

\def\resp{{\it resp. }}

\def\QZ{Q\!Z}

\def\xx{{x}}
\def\yy{{y}}
\def\zz{{z}}
\def\ss{{s}}
\def\tt{{t}}

\def\Chiral{\hbox{$M_\chi$}}
\def\Knuth{\hbox{$M_\kappa$}}
\def\Irr{\Sigma}
\def\length#1{\hbox{$|#1|$}}

\begin{document}

\label{firstpage}
\maketitle

\begin{abstract}
Divisibility monoids (\resp Garside monoids) are a natural algebraic generalization of Mazurkiewicz
trace monoids (\resp spherical Artin monoids), namely monoids in which the distributivity of the
underlying lattices (\resp the existence of common multiples) is kept as an hypothesis, but the
relations between the generators are not supposed to necessarily be commutations (\resp be of Coxeter
type). Here, we show that the quasi-center of these monoids can be studied and described similarly,
and then we exhibit the intersection between the two classes of~monoids.\end{abstract}

%\tableofcontents

\section{Introduction}

\noindent The purpose of this paper is to study some possible connections between the classes
of Garside monoids and divisibility monoids, which are natural algebraic generalizations of monoids
involved in two mathematical areas, namely braid theory and trace theory, respectively.

Garside monoids can be used as a powerful tool in the study of automatic structures for groups arising
in topological or geometric contexts. Generalizing Mazurkiewicz trace monoids, divisibility monoids
have been introduced as a mathematical model for the sequential behavior of some concurrent systems
considered in several areas in computer science and in which two sequential transformations
of the form~$ab$ and~$cd$---with~$a,b,c,d$ viewed as atomic transitions---can give rise to the same
effect.

Here we show that Garside monoids and divisibility monoids are both strangely similar (although they
come from apparently unrelated mathematical theories) and genuinely different (even if they share
some essential algebraic properties). Their similarity will be illustrated by the fact that their
quasi-center---roughly speaking, some supmonoid of the center---can be studied with close techniques
and described through a same statement. Their particularity will be emphasized by the fact that the
intersection between the two classes---that we exhibit by using the result about the
quasi-center---is reduced to a somewhat confined subclass.

\smallbreak
\noindent
{\bf Main Theorem.} {\sl (i) A divisibility monoid is a Garside monoid if and only if every pair of
its irreducible elements admits common multiples.

\noindent (ii) A Garside monoid is a divisibility monoid if and only if the lattice of its simple
elements is a hypercube.}

\smallbreak\noindent The paper is organized as follows. In Section~1, we gather the needed basic
properties of~Garside monoids and~divisibility monoids. In Section~2, we introduce a specific tool
that we call {\it local delta} and which allows us to compute a minimal generating set for the
quasi-center of every divisibility monoid (Propositions~\ref{P:Minimal} and~\ref{P:FreeAbelian}). We
compare these results with those obtained for Garside monoids in~\cite{pib}. In Section~3, we
finally prove the main theorem of this paper (Theorem~\ref{T:intersection}) and illustrate~it.

%%%%%%%%%%%%%%%%%%%%%%%%%%
%%%%%%%%%%%%%%%%%%%%%%%%%%
%%%%%%%%%%%%%%%%%%%%%%%%%% 
\section{Background from Garside and divisibility monoids}

\noindent In this section, we list some basic properties of
Garside monoids and divisibility monoids, and summarize results by Dehornoy \&
Paris about Garside monoids and by Droste \&~Kuske about divisibility monoids.
For all the results quoted here, we refer the reader
to~\cite{dfx,pic,dgk} and to~\cite{dku,kus}.

\subsection{Divisors and multiples in a monoid}
Assume that~$M$ is a monoid. We say that~$M$ is {\it conical} if~$1$ is the
only invertible element in~$M$. For~$a,b$ in~$M$, we say that~$b$ is a left divisor
of~$a$---or that~$a$ is a right multiple of~$b$---if~$a=bd$ holds for some~$d$
in~$M$. The set of the left divisors of~$b$ is denoted by~$\Dar(b)$. An element~$c$
is a right lower common multiple---or a right lcm---of~$a$ and~$b$ if it is a right
multiple of both~$a$ and~$b$, and every right common multiple of~$a$ and~$b$ is a
right multiple of~$c$. Right divisor, left multiple, and left lcm are defined
symmetrically. For~$a,b$ in~$M$, we say that~$b$ {\it divides}~$a$---or that~$b$ is a
divisor of~$a$---if~$a=cbd$ holds for some~$c,d$ in~$M$. 

\medbreak\noindent If~$c$, $c'$ are two right lcm's of~$a$ and~$b$, necessarily~$c$
is a left divisor of~$c'$, and~$c'$ is a left divisor of~$c$. If we assume~$M$ to be
conical and cancellative, we have~$c=c'$~: the unique right lcm of~$a$
and~$b$ is then denoted by~$a\jR b$. If~$a\jR b$ exists, and~$M$ is left cancellative,
there exists a unique element~$c$ satisfying~$a\jR b=ac$~: this element is denoted
by~$a\dR b$. In particular, we obtain the identities~:

\vglue 2mm

\centerline{\fbox{
$a\jR b=a\cdot(a\dR b)=b\cdot(b\dR a).$ }}
\vglue 3mm

\noindent Cancellativity and conicity imply that left and right divisibility are order relations. Now, the additional assumption that any two elements admitting a common multiple admit an lcm---that we will see is satisfied by both Garside monoids and divisibility monoids---allows to obtain the following algebraic properties for the operation~$\dR$.

\begin{lemma}\label{L:Calculous}
Assume that~$M$ is a cancellative conical monoid in which any two elements admitting a common right multiple admit a right lcm. Then the following identities
hold in~$M$~:
\[(ab)\jR(ac)=a(b\jR c),\hbox{\qquad}c\dR(ab)=(c\dR a)((a\dR c)\dR b),\hbox{\qquad} (ab)\dR
c=b\dR(a\dR c),\]
\[(a\jR b)\dR c=(a\dR b)\dR(a\dR c)=(b\dR a)\dR(b\dR c),
\hbox{\qquad} c\dR(a\jR b)=(c\dR a)\jR(c\dR b).\]
For each identity, this means that both sides exist and are equal or that
neither exists.
\end{lemma}

\begin{nota} For some subsets~$A,B$ of elements, we denote by~$A\dR B$ the set of elements~$a\dR b$ provided that all of them exist.
\end{nota}

\subsection{The quasi-center of a monoid}
\noindent The quasi-center of a monoid turns out to be a useful supmonoid of its center.

\begin{defn} Let~$M$ be a monoid. An
{\em irreducible element} of~$M$ is defined to be a non trivial element~$a$ such
that~$a=bc$ implies~either~$b=1$ or~$c=1$. The set of the irreducible elements in~$M$ can be
written as~$(M\setminus\{1\})\setminus(M\setminus\{1\})^2$.
\end{defn}

\begin{defn} Assume that~$M$ is a monoid with set of irreducible
elements~$\Irr$. The {\em quasi-center} of~$M$ is defined to be its submonoid~$\{a\in
M~;~a\Irr=\Irr a\}$.
\end{defn}

\noindent Although straightforward, the two following lemmas capture key properties of quasi-central elements. They will be frequently used in the remaining sections.
  
\begin{lemma}\label{L:Indiff} Assume that~$M$ is a cancellative monoid. Then, for every element~$a$
in~$M$ and every quasi-central element~$b$ in~$M$, the following are equivalent~:

\noindent (i) $a$ divides~$b$;

\noindent (ii) $a$ divides~$b$ on the left;

\noindent (iii) $a$ divides~$b$ on the right.
\end{lemma}

\begin{proof} By very definition, (ii) (\resp (iii)) implies~(i). Now, assume~(i). Then there exist elements~$c,d$ in~$M$ satisfying~$b=cad$. Since~$b$ is quasi-central, we have~$cb=bc'$ for some~$c'$ in~$M$. We find~$cb=bc'=cadc'$, hence, by left cancellation, $b=adc'$, which implies~(ii). Symmetrically, (i) implies~(iii).\end{proof}

\begin{lemma}\label{L:Straight} Assume that~$M$ is a cancellative monoid and~$a$ is a quasi-central element in~$M$. Then, for any two elements~$b,c$ satisfying~$a=bc$, $b$ is quasi-central if and only if so is~$c$.
\end{lemma}

\begin{proof} Assume~$c$ to be quasi-central. Let~$x$ be an irreducible element in~$M$. Since~$a$ is quasi-central, there exists an irreducible element~$x'$ in~$M$ satisfying~$xa=ax'$, hence~$xbc=bcx'$. Now, since~$c$ is quasi-central, there exists an irreducible element~$x''$ in~$M$ satisfying~$cx'=x''c$. We find~$xbc=bx''c$, hence, by right cancellation, $xb=bx''$, which implies that~$b$ is quasi-central. Symmetrically, if~$b$ is quasi-central, so is~$c$.\end{proof}

\subsection{Main definitions and properties for Garside monoids}
\begin{defn}\label{D:GarsideMonoid} A monoid~$M$ is said to be {\em
Garside}\footnote[2]{Garside monoids as defined above are called Garside monoids
in~\cite{cmw,dgk,pif,pie}, but they were called either~"small Gaussian" or~"thin
Gaussian" in previous papers~\cite{dfx,pic,pia,pib}, where a more restricted notion
of~Garside monoid was also considered.} if
$M$ is conical and cancellative, every pair of elements in~$M$ admits a left lcm and a
right lcm, and~$M$ admits a {\em Garside element}, defined to be an element whose
left and right divisors coincide, are finite in number and generate~$M$.
\end{defn}

\begin{exam}\label{E:Garside} All spherical Artin monoids are Garside monoids. Braid
monoids of complex reflection groups~\cite{bmr,pic}, Garside's hypercube monoids~\cite{gar,pic}, Birman-Ko-Lee monoids for spherical Artin groups~\cite{bkl,pif} and monoids for torus link groups in~\cite{pie} are also Garside monoids.

The monoid~$\Chiral$ with presentation
\[\langle~x,y,z: xzxy=yzx^2~,~ yzx^2z=zxyzx~,~ zxyzx=xzxyz~\rangle.\]is a typical example of a Garside monoid, which has the distinguishing feature to be not antiautomorphic.

The monoid~$\Knuth$ defined by the presentation~$\langle~\xx,\yy:\xx\yy\xx\yy\xx\yy\xx=\yy\yy~\rangle$ is another example of a Garside monoid, which, as for it, admits no additive norm, \ie, no norm~$\nu$
satisfying~$\nu(ab)=\nu(a)+\nu(b)$ (a {\em norm} for a monoid~$M$ is a mapping~$\mu$ from~$M$ to~$\mathbb{N}$ satisfying~$\mu(a)\!>\!0$ for every~${a\not=1}$ in~$M$, and
satisfying~$\mu(ab)\geq\mu(a)+\mu(b)$ for every~$a,b$ in~$M$).
\end{exam}

\noindent Every element in a Garside monoid has finitely many
left divisors, only then, for any two elements~$a,b$, the left common divisors of~$a$
and~$b$ admit a right lcm, which is therefore the left gcd of~$a$ and~$b$. This
left gcd is denoted by~$a\gL b$.

Every Garside monoid admits a minimal Garside element, denoted in general
by~$\D$. The set of the divisors of~$\D$---called the {\it simple elements}---endowed with the operations~$\jR$ and~$\gL$ is a finite lattice.

\begin{exam}\label{E:Lattice}
The lattices of simple elements of monoids~$\Chiral$  and~$\Knuth$ of~Example~\ref{E:Garside} are displayed in~Figure~\ref{F:CK}, using Hasse diagrams, where clear
(\resp middle, dark) edges represent the irreducible element~$\xx$ (\resp $\yy$, $\zz$). 
\end{exam}

\begin{figure}%[htb]
\begin{center}
\includegraphics{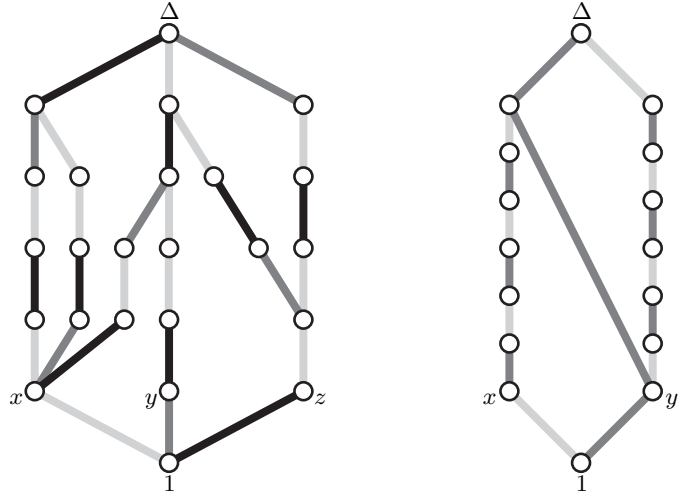}%[width=11cm,height=13mm]
	\put(-187,-4){\makebox(0,0){$1$}}
	\put(-245,28){\makebox(0,0){$\xx$}}
	\put(-194,27){\makebox(0,0){$\yy$}}
	\put(-130,28){\makebox(0,0){$\zz$}}
	\put(-187,175){\makebox(0,0){$\Delta$}}
	\put(-31,-4){\makebox(0,0){$1$}}
	\put(-66,28){\makebox(0,0){$\xx$}}
	\put(3,27){\makebox(0,0){$\yy$}}
	\put(-31,175){\makebox(0,0){$\Delta$}}
\end{center}
\caption{The lattice of simple elements of the Garside monoids~$\Chiral$ (left) and~$\Knuth$ (right).}
\label{F:CK}
\end{figure}

\noindent It is easy to check that every Garside element is a
quasi-central element. Now, this is far from being sufficient to describe what the
quasi-center of every Garside monoid looks~like. The latter was done in~\cite{pib},
where the following structural result was established.

\begin{thm}\label{T:Pib}
Every Garside monoid is an iterated crossed product of Garside monoids with an
infinite cyclic quasi-center.
\end{thm}

%%%%%%%%%%%%%%%%%%%%%%%%%%
%%%%%%%%%%%%%%%%%%%%%%%%%%
%%%%%%%%%%%%%%%%%%%%%%%%%% 
\subsection{Main definitions and properties for divisibility monoids}
\begin{defn}\label{D:DiviMonoid} A monoid~$M$ is called a {\em left divisibility monoid} or simply a
{\em divisibility monoid} if~$M$ is cancellative and finitely generated by its irreducible elements, if
any two elements admit a left gcd and if~every element~$a$ dominates a finite\footnote[4]{This
finiteness requirement is in fact not necessary since it follows from the other stipulations (see e.g.~\cite{kus}).}
distributive lattice~$\Dar(a)$.
\end{defn}

\noindent Note that cancellativity and the lattice condition imply conicity. Like in
the Garside case, the left gcd of two elements~$a,b$ will be
denote by~$a\gL b$. The {\it length}~$\length{a}$ of an element~$a$
is defined to be the height of the lattice~$\Dar(a)$.

\begin{exam} Every finitely generated trace monoid is a divisibility monoid. The
monoids $\langle~\xx,\yy,\zz:\xx\yy=\yy\zz~\rangle$, $\langle~\xx,\yy,\zz:\xx^2=\yy\zz~\rangle$
and~$\langle~\xx,\yy,\zz:\xx^2=\yy\zz,~\yy\xx=\zz^2~\rangle$ are not trace but divisibility monoids. The
monoid~$\langle~\xx,\yy,\zz:\xx^2=\yy\zz,~\xx\yy=\zz^2\rangle$ is not a divisibility
monoid (but a Garside monoid!)~; indeed, the lattice~$\Dar(x^3)$ is not distributive.
\end{exam}

\noindent An easy but crucial fact about divisibility monoids is the
following.

\begin{lemma}\label{L:Lcm}Assume that~$M$ is a divisibility monoid. Then finitely many elements in~$M$ admitting at least a right common multiple admit a unique right lcm.
\end{lemma}

\noindent The following result states that there exists a decidable class of
presentations that gives rise precisely to all divisibility monoids.

\begin{thm}\label{T:Kuske}
Assume that~$M$ is a monoid finitely generated by the set~$\Irr$ of its irreducible
elements. Then~$M$ is a left divisibility monoid if and only if 

\noindent(i) $\Dar(xyz)$ is a distributive lattice, 

\noindent(ii) $xyz=xy'z'$ or~$yzx=y'z'x$ implies~$yz=y'z'$, 

\noindent(iii) $xy=x'y'$, $xz=x'z'$ and~$y\not=z$ imply~$x=x'$,

\noindent for any~$x,y,z,x',y',z'$ in~$\Irr$, and if

\noindent(iv) we have~$M\cong\Irr^*\!/\!\!\sim$, where~$\sim$ is the congruence
on~$\Irr^*$ generated by the pairs~$(xy,zt)$ with~$x,y,z,t$ in~$\Irr$ satisfying~$xy=zt$.
\end{thm}

%%%%%%%%%%%%%%%%%%%%%%%%%%
%%%%%%%%%%%%%%%%%%%%%%%%%%
%%%%%%%%%%%%%%%%%%%%%%%%%% 
\section{The quasi-center of a divisibility monoid}

\noindent This section deals only with divisibility monoids, and no knowledge of
Garside monoids is needed---except possibly for the final remark. After defining the
{\it local delta}, we use it as a fundamental tool in order to give a generating
set for the quasi-center of every divisibility monoid and show that it is minimal. We
establish then that the quasi-center of every divisibility monoid is a free Abelian
submonoid.

\subsection{A generating set for the quasi-center}
We first introduce a partial version of what the author called {\it local
delta} in~\cite{pib}.

\begin{defn}
Assume that~$M$ is a divisibility monoid. Let~$a$ be an element in~$M$. If the
set~$\{b\dR a~;~b\in M\}$ is well-defined\footnote{that is, the element~$b\dR a$ exists---or equivalently, the element~$a\jR b$ exists---for every~$b$ in~$M$.} and admits a right lcm, the latter is
called the {\em right local delta of~$a$} or simply  the {\em local delta of~$a$} and is denoted by
\vglue 2mm
\centerline{\fbox{
$\zR{a}=\jjR\{b\dR a~;~b\in M\}.$
}}
\vglue 2mm

\noindent Otherwise, we say that~$a$ does not admit a right local delta. Note that the equality~$1\dR
a=a$ implies~$a$ to be a left divisor of~$\zR{a}$, whenever it exists, and that, having~$b\dR 1=1$ for every~$b$ in~$M$, we
obtain~$\zR{1}=1$.
\end{defn}

\begin{lemma}\label{L:Key} Assume that~$M$ is a divisibility monoid. Then an element~$a$ in~$M$ admits a local delta if and only if $a\jR b$ exists for every element~$b$ in~$M$.
\end{lemma}

\begin{proof} Assume that~$M$ is a divisibility monoid with~$\Irr$ the set of its irreducible elements and~$\Irr_1=\Irr\cup\{1\}$. For every~$a$ in~$M$, we define~$\Upsilon_i(a)$ by~$\Upsilon_0(a)=\{a\}$ and~$\Upsilon_i(a)=\Irr_1\dR\Upsilon_{i-1}(a)$ for~$i>0$ whenever both~$\Upsilon_{i-1}(a)$ exists and~$c\dR b$ exists for every~$c$ in~$\Irr_1$ and every~$b$ in~$\Upsilon_{i-1}(a)$.

Assume that~$a$ is an element such that~$a\jR b$ exists for every~$b$ in~$M$. From the distributivity of~$\Dar(a\jR b)$, we can deduce~$\length{b\dR
a}\leq\length{a}$. The set~$\{b\dR a~;~b\in M\}$ is then finite, and 
there exists a positive integer~$i_a$ satisfying
\[\{a\}=\Upsilon_0(a)\subsetneq\Upsilon_1(a)\subsetneq\ldots\subsetneq\Upsilon_{i_a}(a)
=\Upsilon_{i_a+1}(a)=\ldots=\{b\dR a~;~b\in M\}.\]
Indeed, by Lemma~\ref{L:Calculous}, we have~$M\dR a=\Irr_1^{i_a}\dR a=\Irr_1\dR(\Irr_1^{i_a-1}\dR a)$. Now, assume~$c,d$ belonging to~$M\dR a$ and~$h$ satisfying~$d=h\dR a$. By hypothesis, the element~$(hc)\dR a$ exists, hence, by Lemma~\ref{L:Calculous}, the element~$c\dR d$---which is~$c\dR (h\dR a)$---exists, so does~$c\jR d$. Therefore, the finite set~$\{b\dR a~;~b\in M\}$ admits a right lcm, namely~$\zR{a}$.

The converse implication is straightforward.\end{proof}

\begin{rem}\label{R:Upsilon}Following the previous proof, we can compute the right local delta of any element~$a$ by recursively computing the sets~$\Upsilon_i(a)$ whenever they exist and then by computing the right lcm of~$\Upsilon_{i_a}(a)$, which must exist by Lemma~\ref{L:Key}. See Examples~\ref{E:Upsilon} and~\ref{E:RankThree} below.
\end{rem}

\noindent We are going to prove that the quasi-center of any divisibility monoid is
generated by the (possibly empty) set of the local delta of its irreducible elements
(Proposition~\ref{P:QuasiOfDiv}). The proof of this result relies on several
preliminary statements.

\begin{lemma}\label{L:zRdivides}
Assume that~$M$ is a divisibility monoid. Then, for every element~$a$ and every
quasi-central element~$b$ in~$M$,  $a$ dividing~$b$ implies~$\zR{a}$
existing and dividing~$b$.
\end{lemma}

\begin{proof} By hypothesis and Lemma~\ref{L:Indiff}, there exists an element~$d$ in~$M$ satisfying~$b=ad$.
As~$b$ is quasi-central, for every~$c$ in~$M$, there exists an element~$c'$ in~$M$
satisfying~$cb=adc'$. By Lemma~\ref{L:Lcm}, for every~$c$ in~$M$, $c\jR a$---which
exists and is~$c(c\dR a)$---divides~$cb$ on the left, and, by left cancellation,
$c\dR a$ exists and divides~$b$ on the left. Therefore, $\zR a$ exists (by Lemma~\ref{L:Key}) and divides~$b$ on the left.\end{proof}

\noindent A weaker but convenient version of Lemma~\ref{L:zRdivides} is the following~:

\begin{lemma}\label{L:zRdivides2}
Assume that~$M$ is a divisibility monoid. Then every quasi-central element~$a$ in~$M$
satisfies~$\zR{a}=a$.
\end{lemma}

\noindent The converse assertion is as follows~:

\begin{lemma}\label{L:zRminQuasi} Assume that~$M$ is a divisibility monoid. Then every
element~$a$ in~$M$ satisfying~$\zR{a}=a$ is quasi-central.
\end{lemma}

\begin{proof} Let~$\xx$ be an irreducible element of~$M$. From~$\jjR(M\dR a)=a$, we
deduce that~$\xx\dR a$ is a left divisor of~$a$. Therefore, $\xx(\xx\dR a)$ is a left
divisor of~$\xx a$. Now, by definition, $\xx(\xx\dR a)$ is~$a(a\dR\xx)$, so there
exists~$d$ in~$M$ satisfying~$\xx a=ad$. We obtain~$\length{\xx a}=\length{ad}$,
hence~$\length{\xx}+\length{a}=\length{a}+\length{d}$. We deduce~$\length{d}=\length{\xx}=1$, thus~$d$ is an irreducible element of~$M$. So,
there exists a mapping~$f_a$ from the irreducible elements of~$M$ into themselves
such that~$\xx a=af_a(\xx)$ holds for every irreducible element~$\xx$. By
cancellativity, $f_a$ is injective, hence surjective~: $a$ is quasi-central by
definition.\end{proof}

\begin{prop}\label{P:Quasi} Assume that~$M$ is a divisibility monoid. Then, for
every~$a$ in~$M$, the element~$\zR{a}$ is quasi-central, whenever it exists.
\end{prop}

\begin{proof}Let~$a$ be an element in~$M$ such that~$\zR{a}$ exists. We claim
that~$\zR{\zR{a}}$ exists and is~$\zR{a}$. By hypothesis, the set~$\{b\dR a~;~b\in
M\}$ is well-defined and admits a right lcm, namely~$\zR{a}$. As~$\Dar(\zR{a})$ is
finite (by definition of a divisibility monoid), so is the set~$\{b\dR a~;~b\in M\}$.
Therefore, $M$ admits a finite subset~$T$ satisfying~$M\dR a=T\dR a$.
Let~$T=\{c_1,\ldots,c_r\}$. For every element~$b$ in~$M$, the element~$((c_1b)\dR a)\jR\cdots\jR((c_rb)\dR
a)$ exists and divides~$\zR{a}$ on the left. By using Lemma~\ref{L:Calculous}, we find
\[
((c_1b)\dR a)\jR\cdots\jR((c_rb)\dR a)=b\dR((c_1\dR a)\jR\cdots\jR(c_r\dR a))=b\dR\zR{a},
\]which implies that~$b\dR\zR{a}$ exists and divides~$\zR{a}$ (on the left) for
every~$b$ in~$M$. We deduce that~$\zR{\zR{a}}$ exists (by Lemma~\ref{L:Key}) and divides~$\zR{a}$. Now, $\zR{a}$ dividing~$\zR{\zR{a}}$,
cancellativity and conicity imply~$\zR{\zR{a}}=\zR{a}$. Therefore, by
Lemma~\ref{L:zRminQuasi},
$\zR{a}$ is quasi-central.\end{proof}

\begin{cor}\label{C:Quasi} Assume that~$M$ is a divisibility monoid. Then the partial application~$a\mapsto\zR{a}$ is a surjection from~$M$ onto~the quasi-center of~$M$.
\end{cor}

\begin{prop}\label{P:QuasiOfDiv}
Assume that~$M$ is a divisibility monoid with~$\Irr$ the set of its irreducible elements.
Then $\{\zR{\xx}~;~\xx\in\Irr\}$ is a generating set of~the
quasi-center~of~$M$.
\end{prop}

\begin{proof} Let~$b$ be a quasi-central element in~$M$. We show using induction
on the length~$\length{b}$ of~$b$ that there exist an integer~$n$ and
irreducible elements~$\xx_1,\ldots,\xx_n$ satisfying~$b=\zR{\xx_1}\cdots\zR{\xx_n}$.
For~$\length{b}=0$, $n$ is~0. Assume now~$\length{b}>0$. Then there exist an
irreducible element~$\xx$ and an element~$b'$ in~$M$ satisfying~$b=\xx b'$. By
Lemma~\ref{L:zRdivides}, $\zR{\xx}$ exists and we have~$b=\zR{\xx}b''$ for some~$b''$
in~$M$ with~$\length{b''}<\length{b}$. By Proposition~\ref{P:Quasi}, the
element~$\zR{\xx}$ is quasi-central, hence, by Lemma~\ref{L:Straight}, so is~$b''$. By induction hypothesis, there exist an integer~$m$ and irreducible elements~$\yy_1,\ldots,\yy_m$ admitting local
delta and satisfying~$b''=\zR{\yy_1}\cdots\zR{\yy_m}$. We
obtain~$b=\zR{\xx}\zR{\yy_1}\cdots\zR{\yy_m}$.\end{proof}

\begin{exam}\label{E:Upsilon} Let us consider the following divisibility monoids~$M_1=\langle~\xx,\yy,\zz:\xx\yy=\yy\zz,~\yy\xx=\zz\yy~\rangle$ and~$M_2=\langle~\xx,\yy,\zz:\xx\zz=\yy\xx,~\yy\zz=\zz\xx~\rangle$. In both monoids, by Lemma~\ref{L:Key}, the irreducible elements~$x$ and~$z$ clearly do not admit right local delta. On the one hand, the quasi-center of~$M_1$ is thus generated by the single element~$\zR{\yy}=\jjR\{\yy\}=\yy$. On the other hand, let us try to compute the local delta of~$\yy$. According to Remark~\ref{R:Upsilon}, 
we compute~$\Upsilon_0(\yy)=\{\yy\}$, $\Upsilon_1(\yy)=\{1,\xx,\yy,\zz\}\dR\{\yy\}=\{1,\xx,\yy,\zz\}$. Now, since $\xx\jR\zz$ does not exist, $\Upsilon_2(\yy)$ cannot exist either. Therefore, no irreducible element admits local delta~: $\{\zR{\ss}~;~\ss\in\{\xx,\yy,\zz\}\}$ is
empty and the quasi-center of~$M_2$ is then trivial.\end{exam}

\subsection{Minimality of the generating set}
We now prove that the generating set given in Proposition~\ref{P:QuasiOfDiv}
is minimal. 

\begin{lemma}\label{L:Minimality}Assume that~$M$ is a divisibility monoid. Then
any two irreducible elements~$\xx,\yy$ in~$M$ admitting local delta satisfy either~$\zR{\xx}=\zR{\yy}$ \hbox{or~$\zR{\xx}\gL\zR{\yy}=1$.}
\end{lemma}

\begin{proof} We first prove that, for any two irreducible elements~$\xx,\yy$ admitting local delta
and  every~$b$ in~$M$, $\zR{\xx}=b\zR{\yy}$ implies~$b=1$. According to Proposition~\ref{P:Quasi},  $\zR{\xx}$ and~$\zR{\yy}$ are quasi-central, and, by Lemma~\ref{L:Straight}, $b$ is also quasi-central. Now, as~$1\dR\xx=\xx$ holds, $\xx$
divides~$\zR{\xx}$, and, by Lemma~\ref{L:Indiff}, we have~$\zR{\xx}=d\xx$ for some~$d$ in~$M$. We
find\[\zR{\xx}=d\xx=b\zR{\yy},\] and, by left cancellation, the element~$d\dR b$---which exists by
Lemma~\ref{L:Lcm}---divides~$\xx$~:
since~$\xx$ is irreducible, $d\dR b$ is either~$\xx$ or~$1$. Assume~$d\dR b=\xx$. Then, $b$ being quasi-central,
Lemma~\ref{L:zRdivides2} implies~$\zR{b}=b=\jjR(M\dR b)$, and, therefore, $\xx$ divides~$b$. By
Lemma~\ref{L:zRdivides}, $\zR{\xx}$ divides~$b$, which, by cancellativity and conicity,
implies~$\zR{\yy}=1$, a contradiction. We deduce~$d\dR b=1$. We find then~$\zR{\yy}=(b\dR d)\xx$, and, by Lemmas~\ref{L:Indiff} and~\ref{L:zRdivides}, $\zR{\xx}$ divides~$\zR{\yy}$, which, by cancellativity and conicity,
implies~$b=1$. 

Finally, let~$\xx,\yy$ be irreducible elements in~$M$ admitting local delta.
Assume~$\zR{\xx}\gL\zR{\yy}\not=1$. Then there exists an irreducible element~$\zz$ in~$M$ dividing
both~$\zR{\xx}$ and~$\zR{\yy}$. By Lemma~\ref{L:zRdivides}, $\zR{\zz}$ divides both~$\zR{\xx}$
and~$\zR{\yy}$, which, by the result above, implies~$\zR{\xx}=\zR{\zz}=\zR{\yy}$.\end{proof}

\begin{prop}\label{P:Minimal} Assume that~$M$ is a divisibility monoid with~$\Irr$ the set of its
irreducible elements. Then $\{\zR{\xx}~;~\xx\in\Irr\}$ is a minimal generating set of~the
quasi-center~of~$M$.
\end{prop}

\begin{proof} By Proposition~\ref{P:QuasiOfDiv}, the set~$\{\zR{\xx}~;~\xx\in
\Irr\}$ generates the quasi-center of~$M$. Let~$\xx$ be an irreducible element such
that~$\zR{\xx}$ exists, and~$a,b$ be quasi-central elements in~$M$. It suffices to show
that~$\zR{\xx}=ab$ implies either~$a=1$ or~$b=1$. Assume~$a\not=1$. Then we
have~$a=\yy a'$ for some irreducible element~$\yy$ and some~$a'$ in~$M$. As~$a$ is
quasi-central, by Lemma~\ref{L:zRdivides}, $\zR{\yy}$ exists and is a left divisor
of~$a$, and, therefore,
$\zR{\yy}$ is a left divisor of~$\zR{\xx}$. We have~$\zR\yy\not=1$, 
hence, by Lemma~\ref{L:Minimality}, $\zR\yy=\zR{\xx}$. Cancellativity and
conicity imply then~$b=1$.\end{proof}

\subsection{A free Abelian submonoid}
We conclude with the observation that the quasi-center of every
divisibility monoid is a free Abelian submonoid (that is, a monoid isomorphic to a direct product of copies of the monoid~$\mathbb{N}$).

\begin{lemma}\label{L:zdOfLcms} Assume that~$M$ is a divisibility monoid.
Let~$a,b$ be elements~in~$M$ admitting local delta. Then $a\jR b$ exists and admits a
local delta, namely~$\zR{a}\jR\zR{b}$.
\end{lemma}

\begin{proof} Let~$a,b$ be elements~in~$M$ admitting local delta. First, by Lemma~\ref{L:Key}, $a\jR b$ exists. Next, by Lemma~\ref{L:Key} again, the element~$b\jR c$ exists for every element~$c$ in~$M$, so does~$a\jR (b\jR c)$. Associativity of the operation~$\jR$ implies that~$(a\jR b)\jR c$ exists for every element~$c$ in~$M$. Therefore, by Lemma~\ref{L:Key}, $\zR{a\jR b}$ exists. Now, by Lemma~\ref{L:Calculous}, $(c\dR a)\jR(c\dR b)=c\dR(a\jR b)$ holds for every~$c$ in~$M$. We obtain then~$\zR{a}\jR\zR{b}=\zR{a\jR b}$.\end{proof}

\begin{prop}\label{P:FreeAbelian}
Assume that~$M$ is a divisibility monoid. Let~$\QZ$ be its
quasi-center. Then~$\QZ$ is a free Abelian submonoid of~$M$, and the
partial function~$a\mapsto\zR a$ is a partial surjective
semilattice homomorphism
from~$(M,\jR)$ onto~$(\QZ,\jR)$.
\end{prop}

\begin{proof} Let~$\Irr$ be the set of the irreducible elements in~$M$. By
Proposition~\ref{P:Minimal}, $\{\zR\xx~;~\xx\in \Irr\}$ is the minimal generating
set of~$\QZ$. So, in order to prove that~$\QZ$ is free
Abelian, it suffices to show that, for
any two~$\xx,\yy$ in~$\Irr$ admitting local delta with~$\zR\xx\!\not=\!\zR\yy$, the
element~$\zR\xx\dR\zR\yy$---which exists by Lemma~\ref{L:zdOfLcms}---is~$\zR\yy$.
Let~$\xx,\yy$ be irreducible elements admitting local delta
with~$\zR\xx\!\not=\!\zR\yy$. Then, by Lemma~\ref{L:Minimality}, we
have~\hbox{$\zR\xx\dR\zR\yy\not=1$.} Now, $\zR\xx\dR\zR\yy$ divides~$\zR\yy$, and, by
Lemma~\ref{L:zdOfLcms}, the element~$\zR\xx\dR\zR\yy$ is quasi-central, which
implies~$\zR\xx\dR\zR\yy=\zR\yy$ by Proposition~\ref{P:Minimal}. The second part of
the assertion follows then from~Lemma~\ref{L:zdOfLcms}.\end{proof}

\noindent Both Propositions~\ref{P:Minimal} and~\ref{P:FreeAbelian} hold for every Garside monoid. The
corresponding proofs however are different, even though the current approach closely follows~\cite{pib}. The point is
that, in the Garside case, the function~$a\mapsto\zR{a}$ is total and coincide with its left
counterpart~$a\mapsto\zL{a}$, but there need not exist an additive length (see for instance
the monoid~$\Knuth$ in Example~\ref{E:Garside}), while, in the divisibility case, there exists an additive length but the
function~$a\mapsto\zR{a}$ is only partial and does not admit a left counterpart in general.

%%%%%%%%%%%%%%%%%%%%%%%%%%
%%%%%%%%%%%%%%%%%%%%%%%%%%
%%%%%%%%%%%%%%%%%%%%%%%%%% 
\section{The intersection}

\noindent We can finally state~:

\begin{thm}\label{T:intersection}(i) A divisibility
monoid is a Garside monoid if and only if every pair of its irreducible elements admits
common multiples.

\noindent (ii) A Garside monoid is a divisibility monoid if and only if the lattice of
its simple elements is a hypercube.
\end{thm}

\begin{proof}(i) The condition is necessary by
very definition of a Garside monoid. We have to show that it is also sufficient.
Let~$M$ be a divisibility monoid with~$\Irr$ the finite set of its irreducible elements.
Assume that every pair of irreducible elements in~$M$ admits right multiples. Then, by
Lemma~\ref{L:Lcm}, every pair of irreducible elements in~$M$ admits a unique right lcm, and,
by Theorem~\ref{T:Kuske}~({\it iv}) or simply by distributivity, such an lcm is of length~2. From this, a straightforward
induction shows that any two elements~$a,b$ in~$M$ admit a unique right
lcm. Therefore, by Lemma~\ref{L:Key}, $\zR a$
exists for every~$a$ in~$M$. In particular, $\zR\xx$ exists for every
irreducible element~$\xx$ of~$M$. Finally, by
Proposition~\ref{P:FreeAbelian}, $\jjR_{x\in\Irr}\zR{x}$ is a minimal Garside element
for~$M$, thus~$M$ is a Garside monoid.

(ii) Let~$M$ be a Garside monoid with~$\Irr$ the finite set of its irreducible elements
and~$\D$ its minimal Garside element. Assume first that~$M$ is divisibility monoid.
Then the finite lattice of its simple elements~$\Dar(\D)$ is distributive. By
Proposition~\ref{P:FreeAbelian}, we have~$\D=\jjR_{x\in\Irr}\zR{x}$. Now, by
Theorem~\ref{T:Kuske}~({\it iv}), for every~$\xx$ in~$\Irr$, $M\dR\xx$ is a subset
of~$\Irr_1=\Irr\cup\{1\}$ including~$\xx$, say~$M\dR\xx=\Irr_{\xx}\cup\{\xx\}$, so we
have~$\zR{x}=\jjR_{y\in\Irr_x\cup\{\xx\}\subset\Irr_1}y$,
hence~$\D=\jjR_{x\in\Irr}\zR{x}=\jjR_{x\in\Irr}x$. Therefore, the lattice~$\Dar(\D)$
is a finite distributive lattice, whose upper bound is the join of its atoms,
so~$\Dar(\D)$ is a hypercube (see~\cite{brk} or for instance~\cite[page 107]{sta}).

Conversely, assume that the lattice~$\Dar(\D)$ is a hypercube, that is, $\Dar(\D)$ is isomorphic to~${\bf
1}^{|\Irr|}$ with~${\bf 1}$ the rank~$1$ chain. According to~\cite[Proposition 3.12]{pia}, for every positive integer~$k$, every element~$a$ dividing~$\D^k$ admits a unique decomposition as~$a=b_1\jR\ldots\jR b_k$ with~$b_i$ $\jR$-indecomposable for~$1\leq i\leq k$, that is, $b_i=b'_i\jR b''_i$ implying either~$b_i=b'_i$ or~$b_i=b''_i$. Then, for every positive integer~$k$,
$\Dar(\D^k)$ is isomorphic to~${\bf k}^{|\Irr|}$ with~${\mathbf k}$ the rank~$k$ chain, which
is a distributive lattice. Now, for every element~$a$ in~$M$, the lattice~$\Dar(a)$ is a sublattice of a lattice~$\Dar(\D^k)$ for some positive integer~$k$. Every sublattice of a distributive lattice being distributive, the lattice~$\Dar(a)$ is thus distributive for every element~$a$ in~$M$.
Therefore, all the requirements in the definition of a divisibility monoid are gathered.\end{proof}

\def\NN{\hbox{$\mathbb N$}}
\def\KB{\hbox{$\mathbb K$}}

\begin{exam}\label{E:RankThree} Up to isomorphism, there are 2 (\resp 5, 23) Garside divisibility
monoids with rank~2 (\resp 3, 4). Those with rank~2
are~$\NN^2=\langle~\xx,\yy:\xx\yy=\yy\xx~\rangle$
and~$\KB=\langle~\xx,\yy:\xx^2=\yy^2~\rangle$\footnote[7]{$\KB$ is for Klein Bottle.}.
Figure~\ref{F:Rank3} (to be compared with Figure~\ref{F:CK}) shows the lattice of simple elements of each of the~5 Garside
divisibility monoids with rank~3. Note that, by using Theorem~\ref{T:Pib}, we can recognize (from the left to the right) the monoids~$\NN^3$, $\NN\bowtie\NN^2$, $\NN\times\KB$, $\NN\bowtie\KB$ and some indecomposable (as a crossed product) monoid. By using~Proposition~\ref{P:FreeAbelian}, we compute that their quasi-center is isomorphic to~$\NN^3$, $\NN^2$, $\NN^2$, $\NN^2$ and~$\NN$,
respectively.
\end{exam}

\begin{figure}[htb]
\begin{center}
\includegraphics{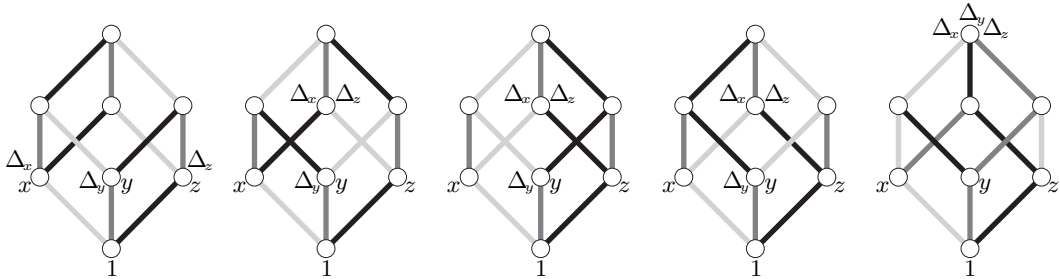}
	\put(-355.5,-4){\makebox(0,0){$1$}}
	\put(-274.3,-4){\makebox(0,0){$1$}}
	\put(-193.3,-4){\makebox(0,0){$1$}}
	\put(-111.9,-4){\makebox(0,0){$1$}}
	\put(-31,-4){\makebox(0,0){$1$}}
	\put(-388.5,27){\makebox(0,0){$\xx$}}
	\put(-307.3,27){\makebox(0,0){$\xx$}}
	\put(-226.3,27){\makebox(0,0){$\xx$}}
	\put(-144.9,27){\makebox(0,0){$\xx$}}
	\put(-64,27){\makebox(0,0){$\xx$}}
	\put(-350,27){\makebox(0,0){$\yy$}}
	\put(-268.8,27){\makebox(0,0){$\yy$}}
	\put(-187.8,27){\makebox(0,0){$\yy$}}
	\put(-106.4,27){\makebox(0,0){$\yy$}}
	\put(-25.5,27){\makebox(0,0){$\yy$}}
	\put(-324,27){\makebox(0,0){$\zz$}}
	\put(-242.8,27){\makebox(0,0){$\zz$}}
	\put(-161.8,27){\makebox(0,0){$\zz$}}
	\put(-80.4,27){\makebox(0,0){$\zz$}}
	\put(0.5,27){\makebox(0,0){$\zz$}}
	\put(-390.5,37){\makebox(0,0){$\zR{\scriptscriptstyle\!\xx}$}}
	\put(-283,62){\makebox(0,0){$\zR{\scriptscriptstyle\!\xx}$}}
	\put(-202,62){\makebox(0,0){$\zR{\scriptscriptstyle\!\xx}$}}
	\put(-120.6,62){\makebox(0,0){$\zR{\scriptscriptstyle\!\xx}$}}
	\put(-40,87){\makebox(0,0){$\zR{\scriptscriptstyle\!\xx}$}}
	\put(-362.9,28.5){\makebox(0,0){$\zR{\scriptscriptstyle\!\yy}$}}
	\put(-281.5,28.5){\makebox(0,0){$\zR{\scriptscriptstyle\!\yy}$}}
	\put(-200.7,28.5){\makebox(0,0){$\zR{\scriptscriptstyle\!\yy}$}}
	\put(-119.3,28.5){\makebox(0,0){$\zR{\scriptscriptstyle\!\yy}$}}
	\put(-30,92){\makebox(0,0){$\zR{\scriptscriptstyle\!\yy}$}}
	\put(-322,37){\makebox(0,0){$\zR{\scriptscriptstyle\!\zz}$}}
	\put(-266,62){\makebox(0,0){$\zR{\scriptscriptstyle\!\zz}$}}
	\put(-184.9,62){\makebox(0,0){$\zR{\scriptscriptstyle\!\zz}$}}
	\put(-103.6,62){\makebox(0,0){$\zR{\scriptscriptstyle\!\zz}$}}
	\put(-21,87){\makebox(0,0){$\zR{\scriptscriptstyle\!\zz}$}}
\end{center}
\caption{The lattices of simple elements of the rank~3 Garside divisibility monoids.}
\label{F:Rank3}
\end{figure}

\noindent In Remark~\ref{R:Upsilon} and Example~\ref{E:Upsilon}, we have seen how local delta in a divisibility monoid are effectively computable. For instance, let us compute the quasi-center of the fifth rank 3 Garside divisibility monoid, namely $M_{3,5}=\langle~\xx,\yy,\zz:\xx^2=\yy\zz,\yy^2=\zz\xx,\zz^2=\xx\yy~\rangle$. We have $\Upsilon_1(\xx)=\{1,\xx,\yy,\zz\}\dR\{\xx\}=\{1,\xx,\zz\}$, $\Upsilon_2(\xx)=\{1,\xx,\yy,\zz\}\dR\{1,\xx,\zz\}=\{1,\xx,\yy,\zz\}=\Irr\cup\{1\}=\Upsilon_3(\xx)=M\dR\xx$. Then~$\xx$ admits a local delta, namely~$\Delta_\xx=\jjR\{1,\xx,\yy,\zz\}=\xx^3$. Symmetrically, we have~$\Delta_\yy=\Delta_\zz=\xx^3$. Therefore, the quasi-center of~$M_{3,5}$ is quite isomorphic to~$\NN$.

\label{lastpage}

\end{document}